\documentclass[12pt, reqno]{amsart}
\usepackage {amsmath,amssymb,verbatim}
\usepackage[pagebackref,pdftex,hyperindex]{hyperref}

\usepackage{lineno}

\newcommand*\patchAmsMathEnvironmentForLineno[1]{%
  \expandafter\let\csname old#1\expandafter\endcsname\csname #1\endcsname
  \expandafter\let\csname oldend#1\expandafter\endcsname\csname end#1\endcsname
  \renewenvironment{#1}%
     {\linenomath\csname old#1\endcsname}%
     {\csname oldend#1\endcsname\endlinenomath}}%
\newcommand*\patchBothAmsMathEnvironmentsForLineno[1]{%
  \patchAmsMathEnvironmentForLineno{#1}%
  \patchAmsMathEnvironmentForLineno{#1*}}%
\AtBeginDocument{%
\patchBothAmsMathEnvironmentsForLineno{equation}%
\patchBothAmsMathEnvironmentsForLineno{align}%
\patchBothAmsMathEnvironmentsForLineno{flalign}%
\patchBothAmsMathEnvironmentsForLineno{alignat}%
\patchBothAmsMathEnvironmentsForLineno{gather}%
\patchBothAmsMathEnvironmentsForLineno{multline}%
}

\usepackage[margin=3cm]{geometry}

\newcommand{\N}{\mathbb{N}}

\renewcommand{\P}{\mathbb{P}}
\newcommand{\Q}{\mathbb{Q}}
\newcommand{\R}{\mathbb{R}}

\newcommand{\cH}{\mathcal{H}}

\newcommand{\cJ}{\mathcal{J}}

\newcommand{\cO}{\mathcal{O}}

\renewcommand{\a}{\alpha}

\renewcommand{\d}{\delta}
\newcommand{\e}{\varepsilon}

\renewcommand{\phi}{\varphi}

\newcommand{\eg}{{\rm e.g.\ }} 
\newcommand{\ie}{{\rm i.e.\ }}

\renewcommand{\leq}{\leqslant}
\renewcommand{\geq}{\geqslant}
\newcommand{\abs}[1]{\left\lvert#1\right\rvert}

\newcommand{\ac}{\mathit{ac}}
\newcommand{\an}{\mathit{an}}

\renewcommand{\div}{\mathrm{div}}

\DeclareMathOperator{\dd}{\sqrt{-1}\partial\bar{\partial}}

\DeclareMathOperator{\End}{End}
\DeclareMathOperator{\Ent}{Ent}

\DeclareMathOperator{\Ext}{Ext}

\DeclareMathOperator{\Hom}{Hom}
\DeclareMathOperator{\id}{id}

\DeclareMathOperator{\lct}{lct}

\DeclareMathOperator{\mult}{mult}

\DeclareMathOperator{\Proj}{Proj}

\DeclareMathOperator{\rank}{rank}
\DeclareMathOperator{\Ric}{Ric}
\DeclareMathOperator{\Rm}{Rm}

\DeclareMathOperator{\supp}{supp}

\DeclareMathOperator{\Tr}{Tr}

\numberwithin{equation}{section}       

\newtheorem{prop} {Proposition} [section]
\newtheorem{thm}[prop] {Theorem} 
\newtheorem{dfn}[prop] {Definition}
\newtheorem{lem}[prop] {Lemma}

\newtheorem{rem}[prop]{Remark}
\newtheorem{setting}[prop] {Setting} 
\theoremstyle{remark}
\newtheorem*{ackn}{\bf{Acknowledgment}}

\newtheorem*{dfn*}{\bf{Definition}}

\title[]{On the Miyaoka-Yau type inequality 
for manifolds with nef anti-canonical line bundle} 
\date{\today} 

\author{Tomoyuki Hisamoto}
\address{Tokyo Metropolitan University\\
Minami-Ohsawa\\
Tokyo\\ 
Japan}
\email{hisamoto@tmu.ac.jp}

\begin{document}

\maketitle

\maketitle
\setcounter{tocdepth}{1}

\begin{abstract}
Based on the recent work of K.~Zhang \cite{Zhang21}, 
we discuss the Miyaoka-Yau type inequality 
for projective manifolds with nef anti-canonical 
line bundle, assuming the lower bound of  
the delta-invariant introduced by Fujita and Odaka. 
\end{abstract}

\tableofcontents

\section{Introduction}
\label{Introduction}

The celebrated Miyaoka-Yau inequality 
restricts the topology 
of complex manifolds 
with ample canonical line bundle 
and has significant role 
in higher dimensional 
algebraic geometry. 
In the two-dimensional case, 
as it was originally shown by \cite{Miyaoka77} and  
\cite{Yau77},  
the same inequality actually holds 
for any smooth minimal surface. 
People tried to extend the inequality 
to the higher dimension. 
The recent paper \cite{Liu20} of W. Liu, 
based on the work \cite{Song20}, \cite{Dy20}, 
showed the Miyaoka-Yau inequality holds 
for any smooth minimal model, \ie 
manifolds with nef canonical 
line bundle. 
The proof itself is evocative, 
as it links the scalar curvature 
with the numerical dimension of the 
canonical line bundle.  
Inspired by his work, 
we will study the anti-canonical setting. 

\begin{thm}\label{Theorem A}
    Let $X$ be an $n$-dimenstional projective manifold 
    with nef anti-canonical line bundle. 
    Let $A$ be an ample line bundle and 
    assume the lower bound 
    of the delta-invariant 
    \begin{equation}\label{assumption of Theorem A}
        \limsup_{\e \to 0}\d (-K_X +\e A) 
        > 1. 
    \end{equation} 
Then the Miyaoka-Yau type inequality 
    \begin{equation}\label{MY}
        \bigg\{ 2(n+1)c_2(X) -n c_1^2(X) \bigg\} c_1(X)^{n-2}\geq 0 
     \end{equation} 
     holds. 
     If the anti-canonical line bundle is nef and big, 
     the assumption (\ref{assumption of Theorem A}) is equivalent to $\d(-K_X)>1$ and  
     the equality of (\ref{MY}) holds if and only if 
     the anti-canonical model 
     admits a finite, codimension-one \'etale cover $\P^n \to X_{ac}$. 
\end{thm}

In the anti-canonical setting, 
the classical result \cite{CO75} asserts that 
the inequality (\ref{MY}) holds 
for any K\"ahler-Einstein Fano manifolds. 
On the other hand, \cite{FO16} introduced the delta-invariant 
and showed that $X$ admits a unique K\"ahler-Einstein metric if $\d(-K_X)>1$. 
Thus we obtain the ample case in Theorem \ref{Theorem A}. 
The result for arbitrary nef anti-canonical line bundles 
possibly be a kind of folklore to the experts, since it  directly 
follows from the recent existence result of the twisted K\"ahler-Einstein metric 
combined with the classical inequality (\ref{Chen-Ogiue}) 
for the Riemannian tensor norm. 
\begin{thm}[Consequence of \cite{Zhang21}, Theorem 2.3]\label{existence of twisted KE}
    Let $\theta$ be a smooth negative $(1, 1)$-form 
    in the cohomology class $-c_1(A)$. 
    Under the assumption of Theorem \ref{Theorem A}, 
    there exists a twisted K\"ahler-Einstein metric 
 $\omega_\e \in  2\pi c_1(-K_X) +\e c_1(A) $ 
 satisfying 
 \begin{equation}
    \Ric(\omega) =\omega +\e \theta. 
 \end{equation} 
\end{thm} 
The point is that the above result allows the 
negative twist $\theta$, compared to the 
previous result \cite{BBJ15}. 

If the anti-canonical line bundle is nef and big, 
one may also obtain a little bit more algebraic proof 
relying on the result of \cite{Xu}. 
The equality condition for (\ref{MY}) will be discussed 
along this line. 

We will give another approach, applying the idea of \cite{Liu20} 
to the variational framework for the constant scalar 
curvature K\"ahler metric. 
It will be succesful if there exists a 
sequence of the cscK (constant scalar curvature K\"ahler) metrics. 

\begin{thm}\label{existence of cscK}
    If each $2\pi c_1(-K_X)+ \e_j c_1(A)$ 
    admits a cscK metric for 
    some sequence $\e_j \to 0$, 
    then the inequality (\ref{MY}) holds. 
    If we assume $\nu(-K_X)=1$ 
    besides the condition in Theorem \ref{Theorem A}, 
    there exists a cscK metric 
 $\omega_\e \in  2\pi c_1(-K_X) +\e c_1(A) $ 
 for any small $\e>0$. 

\end{thm}

It is interesting that 
the both twisted 
K\"ahler Einstein and cscK metric 
ensure the inequality (\ref{MY}). 
To the best of the author's knowledge, 
the second part of the above theorem  
has been not included in the existence results so far, 
\eg \cite{Dervan}, Theorem 1.1, \cite{Zhang21}, Theorem 2.4. 
Some interesting examples, geometrically ruled surfaces and nine points blowing up 
of $\P^2$, will be discussed. 

Following the idea of \cite{Tian92}, 
one may also derive from Theorem \ref{existence of twisted KE} 
the semistability of the holomorphic tangent vector bundle.

\begin{thm}\label{slope semistability}
Assume $-K_X$ is nef and big. 
If $\d(-K_X)>1$, 
The holomorphic tangent bundle $T_X$
and its canonical extenstion sheaf 
is slope semistable 
in the sense of Mumford-Takemoto. 
\end{thm} 

\begin{ackn}
    The author thanks Professor S.~Boucksom 
    for his pointing out in the first version of our manuscript 
    that Theorem A more directly follows 
    from \cite{Zhang21}, Theorem 2.3. 
    He also thanks Professor M.~Iwai 
    for the fruitful discussion about the 
    equality condition in Theorem \ref{Theorem A}. 
    I am also grateful to Peofessor K.~Zhang 
    for his kind and helpful comments. 
\end{ackn}

\section{Delta invariant and K\"ahler-Einstein metric on a Fano manifolds}

In this section 
we explain some preliminary materials 
concerning the canonical energy functionals 
on the space of K\"ahler metrics. 
For a while we fix an ample line bundle $L$ on $X$ and 
take a K\"ahler metric $\omega$ 
such that $[\omega] =c_1(L)$. 
It is called constant scalar curvature K\"ahler 
(in short, cscK) metric if 
the scalarcurvature 
$R(\omega)= \Tr \Ric(\omega)$ 
is constant everywhere. 
It is well-known that in the case 
$L = \pm K_X$ 
cscK is equivalent to 
the K\"ahler-Einstein condition 
$\Ric(\omega) =\mp \omega$. 
Notice that, since $\omega $ is assinged in the 
fixed cohomology class, the volume 
$V:= \int_X \omega^n$ and the mean value 
\begin{equation}
    \hat{R} 
    := \frac{1}{V}\int_X R(\omega) \omega^n 
    =-\frac{nK_XL^{n-1}}{L^{n}}
\end{equation} 
is determined by $c_1(L)$. 

Let us fix $\omega_0$ and 
consider the set 
\begin{equation}
    \cH_{\omega_0} 
    := \bigg\{ 
        \phi \in C^\infty(X: \R): 
        \omega_0 + \dd \phi >0 \bigg\}. 
\end{equation} 
By the famous $\partial \bar{\partial}$-lemma, 
we may identify $\cH_{\omega_0} /\R$ with 
the collection of K\"ahler metrics in $c_1(L)$. 
Each tangent vector $u \in T_\phi\cH_{\omega_0}$ 
is identified with a smooth function on $X$. 
The first important energy 
on this space is Monge-Amp\`ere energy 
$E\colon \cH_{\omega_0}\to \R$ 
which is characterized by the differential 
$(dE)_\phi u =V^{-1}\int_X u \omega_\phi^n$. 
This energy is explicitly written as $E(\phi) 
= (nV)^{-1}\sum_{i=0}^n \int_X \phi 
\omega_0^{i}
\wedge \omega_\phi^{n-i}$ and 
effectively used in the study 
of the complex Monge-Map\`ere equation. 
We refer the textbook \cite{GZ17} for the 
exposition. 
Similarly, for a smooth real $(1, 1)$-form 
$\gamma$ 
one may define the 
twisted Monge-Amp\`ere energy 
$E_\gamma \colon \cH_{\omega_0}\to \R$ 
which satisfies 
\begin{equation}\label{differential}
    (dE_\gamma)_\phi u 
    =\frac{1}{V}\int_X u \gamma \wedge \omega_\phi^{n-1}.  
\end{equation} 
The differential description determines the energy 
up to addition of constant. 
One indeed has the explicit formula 
\begin{equation}\label{twisted MA energy}
    E_\gamma (\phi) 
    = \frac{1}{nV}\sum_{i=1}^n \int_X \phi 
    \gamma \wedge \omega_0^{i-1}
    \wedge \omega_\phi^{n-i}. 
\end{equation} 
A simple integration-by-part computation shows that 
 $E_\gamma$ defined by (\ref{twisted MA energy}) 
 enjoys the property (\ref{differential}). 
In the special case $\gamma =\Ric(\omega_0)$ 
we will write $E_{\Ric}$. 

Since $E(\phi+c)=E(\phi)+c$, $E$ is not 
defined as a fucntion in $\omega=\omega_\phi$. 
For this reason some variant of the Monge-Amp\`ere 
energy is useful in the study of standard metrics. 
One is the simple J-functional 
\begin{equation}
    J(\phi) := \frac{1}{V}\int_X \phi  \omega_0^n 
    -E(\phi), 
\end{equation}
which is invariant under the scaling 
$\phi \mapsto \phi +c$ and a nonnegative functional.  
It is also convenient to introduce 
the symmetric (with respect to $\omega_0$ and $\omega$) I-functional: 
\begin{equation}
I(\phi) 
:= \frac{1}{V} \int_X \phi \omega_0^n 
-\frac{1}{V} \int_X \phi \omega_\phi^n. 
\end{equation} 
Using integration by part 
we observe  
$I \geq 0$ and 
a simple comparison:  
\begin{equation}\label{I vs J}
 \frac{n+1}{n} J \leq I \leq (n+1)J. 
\end{equation} 
It is also introduced by \cite{Don99} 
the modified J-fuctional 
\begin{equation}
    \cJ_\gamma (\phi) 
    :=E_\gamma(\phi) - c_\gamma E(\phi), 
\end{equation} 
where $c_\gamma := V^{-1}[\gamma]L^{n-1} $ is the numerical constant 
such that $\cJ_\gamma(\phi +c)= \cJ_\gamma(\phi)$.  
The cricital point of $\cJ_\gamma$ 
satisfies the so called J-equation 
\begin{equation}\label{J-equation}
    \Tr_{\omega} \gamma = n c_\gamma
\end{equation}  
and 
numerical criterion for the existence of the solution 
has been obtained by \cite{SW08}, \cite{Song20}, and \cite{Chen21}. 

Now we introduce 
the most important K-energy functional 
intorduced by \cite{Mab86}, which is 
characterized by the property 
\begin{equation}
    (dM)_\phi u 
    =-\frac{1}{V}\int_X u (R(\omega)-\hat{R})
     \omega_\phi^n.  
\end{equation} 
It actually defines a geodecially convex function 
on $\cH_{\omega_0}$ and 
a given $\omega$ is a critical point of $M$ 
if and only if it is cscK.  
The following explicit 
formula is due to \cite{Chen00b}. 

\begin{thm}[\cite{Chen00b}. See also \cite{BHJ19}.] 
Denote the relative entropy of the 
two probability measure $\nu, \mu$ by 
\begin{equation}
\Ent(\nu \vert \mu) 
:= \int_X \log{\bigg[  \frac{d\nu}{d\mu} \bigg]} 
d\nu. 
\end{equation} 
Then we have for all $\phi \in \cH_{\omega_0}$ the identity 
\begin{align}\label{K-energy}
    M(\phi) 
    &= \Ent(V^{-1}\omega_\phi^n\vert V^{-1}\omega_0^n) 
    -n E_{\Ric}(\phi) + \hat{R} \cdot E(\phi) \\
    &= \Ent(V^{-1}\omega_\phi^n\vert V^{-1}\omega_0^n) 
    +n \cJ_{-\Ric}(\phi) 
\end{align} 

\end{thm} 

One of the great achievement 
in recent developement of this area is 
the following. 

\begin{thm}\cite{CC21a,CC21b}\label{CC}
If the K-energy is proper, \ie 
$M \geq \d I$ for a positive constant $\d >0$, 
then $(X, L)$ admits a cscK metric. 
\end{thm} 

Based on the work \cite{N90}, 
G. Tian \cite{Tian97} introduced the following 
invariant to ensure the existence of 
a K\"ahler-Einstein metric. 
\begin{dfn}
    Let $L$ be a big line bundle. 
    We define alpha-invariant $\a(L)$ 
    as a supremum of positive constant $a>$ 
    for which some constant $C_a$ ensures 
    the estimate 
    \begin{equation}
\int_X e^{-a (\phi -\sup_X \phi)} \leq C_a
    \end{equation} 
 for all $\phi \in \cH_{\omega_0}$. 
\end{dfn}

We indeed have the following properness result 
for the K-energy. 
\begin{thm}[\cite{BBEGZ11}, Proposition 4,13. See also \cite{BHJ17}, Proposition 9.16.]\label{properness in the Fano case}
   Let $X$ be a complex Fano manifold. 
    If $\a=\a(-K_X) >\a'>\frac{n}{n+1}$, 
    we have properness of the K-energy: 
\begin{equation}
    M \geq \big(\a'- \frac{n}{n+1}\big) I -C. 
\end{equation} 
\end{thm} 

This is a consequence of properness of the entropy: 
\begin{equation}\label{properness of entropy}
\Ent(V^{-1}\omega_\phi^n\vert V^{-1}\omega_0^n) 
\geq \a' I(\phi) -C, 
\end{equation} 
which is 
indeed true for arbitrary ample line bundle $L$ and $\a' <\a(L)$. 
If $L=-K_X$, one has simpler description 
of the K-energy 
$M=\Ent(V^{-1}\omega_\phi^n\vert V^{-1}\omega_0^n)
-n\cJ_{\omega_0}$ so applying (\ref{I vs J}) 
obtains Theorem \ref{properness in the Fano case}. 
we will generalaize the result 
to the nef line bundles 
in the next subsection. 

One can also obtain a purely 
algebraic interpletation of the alpha-invariant: 
\begin{equation}\label{algebraic definition of alpha}
    \a(L) = \sup\bigg\{ \lct(X, D): D \in \abs{L}_\Q \bigg\}. 
\end{equation} 
See \cite{Che08} for the detail and explicit 
computation for specific examples. 
The invariant $\a(L)$ is actually 
determined by the cohomology class $c_1(L)$. 
Thanks to the result in \cite{Dervan}, 
it is also continuous in the K\"ahler cone. 

The delta-invariat introduced in \cite{FO16} 
gives a refinement of the above alpha. 
\begin{dfn}
Let $L$ be a big line bundle. 
A $\Q$-divisor $D$ is called of $k$-basis type if 
there exists a basis $s_1, s_2, \dots, s_{N_k}$ of 
the vector space $H^0(X, kL)$ such that 
\begin{equation}
    D = \frac{1}{kN_k}\sum_{i=1}^{N_k} \div{(s_i)}. 
 \end{equation} 
We define $\d_k(L)$ as the supremum of the log canonical threshold $\lct(X, D)$ 
for $k$-basis type divisor and 
\begin{equation}\label{algebraic definition of delta}
 \d(L) := \limsup_{k \to \infty} \d_k(L). 
\end{equation} 
\end{dfn} 

It is immedeate $\d(L) \geq \a(L)$. 
On the other side \cite{BJ17} showed $\a(L) \geq \frac{1}{n+1}\d(L)$ and that  
if $L$ is ample $\d(L) \geq \frac{n+1}{n}\a(L)$ holds. 
The main result of \cite{FO16} states that 
a Fano manifold $X$ is uniformly K-stable 
(and hence admits a K\"ahler-Einstein metric ) 
if $\d(-K_X)>1 $. 
As it was shown by \cite{BJ17}, 
the condition $\d(-K_X)>1$ is indeed 
equivalent to the uniform K-stability. 
Continuity property of the delta-invariant 
in the big cone is proved by 
\cite{Zhang20}. 

We need more quantative analytic aspect  
in terms of the twisted K-energy functional. 
Let us take a smooth real $(1, 1)$ form $\theta$ 
and consider the ample line bundle 
$L$ such that $c_1(-K_X)=c_1(L)+[\theta]$. 
We fix $\omega_0 \in c_1(L)$ and 
take a function $\rho$ such that 
$\Ric(\omega_0) = \omega_0+\theta +\dd \rho_\theta$ 
with normalization $\int_X e^{\rho_\theta} \omega_0^n=V$.  
for $\phi \in\cH_{\omega_0}$ define the 
twisted K-energy as 

\begin{align}\label{twisted K-energy}
    M_{\theta} (\phi) 
    &=\Ent(V^{-1}\omega_\phi^n \vert V^{-1}\omega_0^n)
    -nE_{\Ric-\theta} +\hat{R}_\theta E \\
    &=\Ent(V^{-1}\omega_\phi^n \vert V^{-1}\omega_0^n)
    +n \cJ_{-\Ric +\theta}(\phi). 
\end{align} 

\begin{thm}[Consequence of \cite{Zhang20}, Proposition $3.6$ and \cite{Zhang21}, Theorem 2.2]\label{properess of the twisted K-energy}
    If $L$ is an ample line bundle and 
    $c_1(-K_X)=c_1(L)+[\theta]$ holds 
    for a (not necessarily positive) smooth form $\theta$. 
    Assume $\d(L)>1$. 
    Then the twisted K-energy (\ref{twisted K-energy}) is proper. 
\end{thm} 

Notice that the convexity of the twisted K-energy 
(\ref{twisted K-energy}) fails if $\theta$ is not positive. 
See \cite{BDL15}, Thoerem 1.2. 
In fact Theorem \ref{existence of twisted KE} exploits the properness of the twisted Ding energy fuctional. 

\section{Proof of Theorem \ref{Theorem A}}
\label{proof}
Let $\omega$ be an arbitrary K\"ahler metric and 
denote its Riemannian, Ricci, and scalar curvature tensor by 
$\Rm(\omega), \Ric(\omega)$, and $R(\omega)$. 
The normalized curvature tensor is defines by 
\begin{equation}
    \widetilde{\Rm}(\omega)
    := R_{i\bar{j}k\bar{\ell}}
    -\frac{R(\omega)}{n(n+1)}
    (g_{i\bar{j}}g_{k\bar{\ell}} +g_{i{\bar{\ell}}}g_{k\bar{j}}), 
    \ \ \ 
    \widetilde{Ric}(\omega)
    := \Ric(\omega) -\frac{R(\omega)}{n}\omega. 
  \end{equation} 

We starts from the fundamental identity 
which is a basis of the Miyaoka-Yau inequaltiy in \cite{CO75}, \cite{Yau77}. 

\begin{prop}[\cite{CO75}]\label{fundamental identity}
    For a K\"ahler manifold $(X, \omega)$ 
     we have the identity 
\begin{align}
    &\bigg\{ 2(n+1)c_2(X) -n c_1^2(X) \bigg\} [\omega]^{n-2} \\
    &=
    \frac{1}{4\pi^2n(n-1)}
    \int_X \bigg\{ (n+1)\abs{ \widetilde{\Rm}(\omega) }^2
    -(n+2)\abs{\widetilde{\Ric}(\omega)}^2 
    \bigg\} \omega^n. 
\end{align} 
\end{prop} 
Noting 
\begin{equation}
    \abs{\widetilde{\Ric}(\omega)}^2
=\abs{\Ric(\omega)}^2 -\frac{R^2}{n}
\leq \abs{\Ric(\omega)-\omega}^2 
\end{equation} 
we obtain the inequality 
    \begin{align}\label{Chen-Ogiue}
        &\bigg\{ 2(n+1)c_2(X) -n c_1^2(X) \bigg\} [\omega]^{n-2} \\
        &\geq 
        \frac{1}{4\pi^2n(n-1)}
        \int_X \bigg\{ (n+1)\abs{ \widetilde{\Rm}(\omega) }^2
        -(n+2)\abs{\Ric(\omega)-\omega}^2 
        \bigg\} \omega^n,  
    \end{align} 
as a corollary of Proposition \ref{fundamental identity}. 
If further $\omega$ is K\"ahler-Einstein metric 
(with positive Ricci curvature), 
the above computation immediately implies the inequality (\ref{MY}). 

Not only that, to show the inequality (\ref{MY}) in more general setting, 
it is sufficient to obtain a sequence of 
K\"ahler metrics $\omega_j$ such that 
\begin{equation}\label{approximately KE}
    \lim_{j \to \infty} [\omega_j] = 2\pi c_1(-K_X), 
    \ \ \ 
    \lim_{j \to \infty} \int_X \abs{\Ric(\omega_j) - \omega_j} =0.   
\end{equation} 

Let us give another approach 
when $-K_X$ is nef and big. 
Assuming the stability condition $\d(-K_X)>1$, 
Theorem 1.1 and 1.2 of \cite{Xu} asserts that 
the anticanonical ring is finitely generated. 
Moreover, the anti-canonical model 
$X_\ac := \Proj \oplus_{k \in \N} H^0(X, -kK_X)$ 
is $\Q$-Fano variety and 
inherits the condition $\d(-K_{X_\ac})>1$. 
It implies that the variety 
admits a singular K\"ahler-Einstein metric. 
Applying \cite{DGP}, Theorem B, 
one may conclude that 
 the canonical extension of the tangent bundle $T_X$ is 
 slope semistable with respect to 
 the nef and big class $c_1(X)$ 
 (see section \ref{stability of the holomorphic tangent bundle} 
 of this article). 
 The slope semistability implies the 
 Bogomolov-Gieseker inequality. 
 Therefore we reprove Theorem \ref{Theorem A},
  in the case $-K_X$ is nef and big. 

Now assume the equality in (\ref{MY}) holds. 
As \cite{GKPT19}, Proposition $8.2$, 
$X_{\ac}$ inherits the equality condition. 
The uniformization result 
\cite{GKP20}, Theorem 1.3 asserts 
that there exists a finite codimension-one \'etale cover 
$\P^n \to X_\ac$. 

\section{Proof of Theorem \ref{existence of cscK}}
\label{proof}

Let us prove 
the first part of Theorem \ref{existence of cscK}. 

\begin{thm}\label{cscK implies MY}
    If each $ 2\pi c_1(-K_X) +\e_j c_1(A)$ admits 
    a cscK metric for some 
    sequence $\e_j \to 0$, then 
    the inequality (\ref{MY}) holds. 
\end{thm} 

\begin{dfn} 
    We define the numerical dimension of 
    a nef line bundle $L$ by 
    \begin{equation}
        \nu(L)
        =\max{\bigg\{k=0, 1, 2, \dots, n 
        \ \ \text{such that} \ \ 
        L^k A^{n-k} \neq 0. \bigg\}}. 
     \end{equation} 
     It does not depend on the choice 
     of an ample divisor $A$. 
    \end{dfn} 

If $L$ is nef and big, 
one has $\nu(L)=n$. 
The following key lemma 
links the scalar curvature  
to the numerical dimension.  

\begin{lem}[\cite{Liu20}, Lemma $2.1$]\label{key lemma}
    Let $L$ be a nef line bundle and 
     $\omega_\e$ be a sequence 
     of 
     K\"ahler metric with cohomology class $ 2\pi c_1(L) +\e c_1(A)$.   
   \begin{equation}
    \lim_{\e \to 0}\frac{ 2\pi n c_1(L)[\omega_\e]^{n-1}}{[\omega_\e]^n}
    = \nu(L)  
   \end{equation}
\end{lem} 
\begin{proof}
    While the proof is totally the same as 
    the case $L=K_X$, 
    we repeat it for the convenience to the reader.
    Take $\eta \in 2\pi c_1(L)$ and 
    $\gamma \in c_1(A)$ as real $(1,1)$-form representatives 
 for each cohomology class. 
 We may assume that $\gamma$ is positive. 
 It is enough to compare 
    \begin{align}
        [\omega_\e]^n 
        =\int_X(\eta +\e \gamma)^n 
        =\sum_{i=0}^n 
        \binom{n}{i} (2\pi)^{i}\e^{n-i} \int_X 
        c_1(L)^i \wedge \gamma^{n-i}
    \end{align} 
  with 
    \begin{align}
         2\pi c_1(L)[\omega_\e]^{n-1}
        =\sum_{i=0}^{n-1}
        \binom{n-1}{i} (2\pi)^{i+1}\e^{n-i-1} \int_X 
        c_1(L)^{i+1} \wedge \gamma^{n-i-1}. 
    \end{align}
 Taking the ratio of these two 
 and letting $\e \to 0$ we obtain the claim. 
\end{proof} 

If $L=-K_X$ and $\omega_\e$ is cscK, 
the above lemma yields 

\begin{equation}
    R(\omega_\e) =\hat{R}_\e 
    =\frac{2\pi n c_1(X)[\omega_\e]^{n-1}}{[\omega_\e]^n}
    \to  \nu(-K_X). 
\end{equation} 

It in particular implies that 
$R(\omega_\e)$ is bounded in $\e$. 

If one computes the asymptotic of 
the Calabi-type functional, 

\begin{align*}
    \int_X \abs{\Ric(\omega_\e) -\omega_\e}^2 \omega_\e^n 
    &=\int_X \bigg\{ \abs{\Ric(\omega_\e)}^2 
    -2R(\omega_\e) +n 
    \bigg\} \omega_\e^n \\
    &=\int_X \bigg\{ \abs{R(\omega_\e)}^2 
    -2R(\omega_\e) +n 
    \bigg\} \omega_\e^n  
    -4\pi^2 n(n-1)c_1^2(X)[\omega_\e]^{n-2} 
\end{align*}
converges to 
\begin{equation}\label{limit}
    ((\nu^2 -2\nu +n)-n(n-1))(2\pi)^n c_1(X)^n.  
\end{equation} 
If $c_1(X)^n=0$ the limit (\ref{limit}) is obviously zero. 
If it is not the case, we have $\nu =n$ 
by the definition of the numerical dimension 
so the limit still vanishes. 
Thus we obtain the sequence $\omega_j :=\omega_{\e_j}$ which 
satisfies the condition (\ref{approximately KE}), 
and conclude Theorem \ref{cscK implies MY}. \\

Now we discuss the second part of Theorem \ref{existence of cscK}. 
It can be compared with the nef canonical bundle case \cite{Song20}, \cite{Dy20}. 

We compute the second and the third terms 
of (\ref{twisted K-energy}), in the case $L_\e=-K_X +\e A$. 
One has to be careful for the dependence of $L_\e$ 
and the space $\cH_{\omega_0}$ 
on $\e$. 
From Theorem \ref{properess of the twisted K-energy}, 
we have the properness of the twisted K-energy: 

\begin{equation}
 M_\theta \geq \lambda \cJ_{\omega_0}. 
\end{equation} 

More precisely, one may take 
$\lambda = \frac{a(d-1)}{d-a}$
with $d \in (1, \d(L; \theta))$ 
and $0< a <\min\{1, \a(L; \theta)\}$. 
In our situation 
$\a(L; \theta)=\a(L)$ and 
$\d(L; \theta) =\d(L) >1$ 
so that  
we obtain the lower bound 
$\lambda \geq \lambda_0$ which is uniform in $\e$. 
To show properness of the ordinal K-energy, 
it is enough to show the properness of 
$\cJ_{\frac{\lambda}{n} \omega_0  -\theta}$. 
From the recent study of the J-equation, 
such properness of the modified J-functional is 
related to J-stability. 

\begin{thm}[\cite{Chen21}. See also \cite{Song20}.]\label{properness of J-energy}
Assume $\gamma$ is a positive form. 
 The modified J-functional $\cJ_\gamma$ 
 is proper 
 if there exists a K\"ahler form 
 $\omega \in c_1(L)$ such that 
 for any $1 \leq p \leq n-1$ and $p$-dimenstional subvariety 
 $V$ Nakai-Moischezon type inequality 
    \begin{equation}\label{NM inequality}
\int_V  (nc_\gamma L-p[\gamma]) L^{p-1} 
\geq  \delta (n-p)\int_V L^p     
    \end{equation} 
holds. 
\end{thm} 

We apply the above theorem to $L=L_\e$ and 
$\gamma \in c_1(L+\e' A )$, with 
$\e' =\frac{n}{\lambda} \e$. 

\begin{prop}\label{check the NM condition}
    Assume $-K_X$ is nef and $\nu(-K_X)=1$. 
    If we choose $L=L_\e$ and $\gamma \in c_1(L+\e' A )$, 
    the inequality (\ref{NM inequality}) holds. 
\end{prop} 
\begin{proof} 
    In the same way as Lemma \ref{key lemma}, 
    we may compute 
\begin{align}
    nc_\gamma 
    &=n \frac{[\gamma]L^{n-1}}{L^n}
    =n \frac{(L_\e +\e'A)L_\e^{n-1}}{L_\e^n}
    = n +\e'n \frac{A(-K_X+\e A)^{n-1}}{(-K_X +\e A)^n}\\
    &=
    \begin{cases} n+O(\e) & \text{if $-K_X$ is big,} \\
        n+\frac{n}{\lambda}(n-\nu(-K_X)) +O(\e) 
        &  \text{otherwise.}
    \end{cases}
\end{align}
The $O(\e)$ term is actually nonnegative. 
It yields  
\begin{align}
    \int_V  (nc_\gamma L-p[\gamma]) L^{p-1}  
    \geq \int_V \bigg\{ (n-p)L +\frac{n}{\lambda}(n-\nu(-K_X) -p)\e A \bigg\}L^{p-1}
\end{align}
and the right-hand side has a desired lower bound 
in the case $\nu(-K_X)=1$. 

\end{proof} 

Applying Theorem \ref{CC}, 
we deduce that each class $c_1(-K_X +\e A)$ admits 
a cscK metric. 
This completes the proof of Theorem \ref{existence of cscK}.

The argument in the proof of Proposition \ref{check the NM condition} 
shows that the condition of Theorem \ref{properness of J-energy} 
fails in the case $\nu(-K_X)=n=2$.

\section{Surface examples}

Let us consider the following two class of 
surfaces. 
\begin{setting}\label{surfaces}
    Let $X$ be a smooth complex 
    projective surface 
    with nef anti-canonical line bundle such that $\nu(-K_X)=1$. 
    We particularly consider one of the following cases. 
    \begin{itemize}
        \item[$(1)$]
        $X$ is a geometrically ruled surface 
        over an elliptic curve $B$. 
        If it is abundant, namely if $\kappa = \nu =1$, 
        after a finite \'etale cover $X \simeq B \times \P^1$
        or $X \simeq \P E$ 
        for some rank two slope semistable vetor bundle $E$. 
        Otherwise $\kappa=0, \nu=1$, and $X \simeq \P E$ 
        holds for some $E$
        which is $(a)$ a nontrivial extension of 
        $\cO$ by itself $Q \simeq \cO$, or 
        $(b)$ direct sum of $\cO$ with 
        some (non-torsion) degree zero 
        line bundle $Q$.  
        \item[$(2)$] 
        $X$ is obtained as a nine points 
        blowing up of $\P^2$. 
        It is abundant if $-K_X$ defines 
        an elliptic pencil. 
        If one choses very general nine points,   
        $\kappa=0, \nu=1$.  
    \end{itemize} 
\end{setting}
From the classification
results for surfaces, 
all examples of non-abundant surfaces are 
contained in the above.

We inquire the lower bound of 
\begin{equation}
\tilde{\a}(L) 
:= \limsup_{k\to \infty} \a(L+\e A), \ \ \ 
\tilde{\delta}(L) 
:= \limsup_{k\to \infty} \delta(L+\e A)
\end{equation}
for $L=-K_X$ in these examples. 
Note that the definition of $\tilde{\a}$ and 
$\tilde{\d}$ might depends on the choice 
of an ample divisor $A$. 
In the nef line bundle case 
we adopt (\ref{algebraic definition of alpha}), 
(\ref{algebraic definition of delta}) 
as the definition of $\a(L), \d(L)$ 
and distinguish $\a$ from $\tilde{\a}$.  
Since the line bundle is only nef, 
it happens $\a(L)>\tilde{\a}(L)$. 
There is yet another definition of 
the alpha invariant 
\begin{equation}
\a^{\an}(L) 
:= \sup\bigg\{ a>0: 
\text{$\exists C_a$ \ $\forall \phi \in \cH_{\omega_0}$ \ \ \ 
$\int_X e^{-a (\phi -\sup_X \phi)} \leq C_a$}
\bigg\}, 
\end{equation}
which is equivalent to $\a(L)$ if $L$ is ample. 
It is easy to check 
\begin{equation}
\a(L) \geq \a^\an(L) \geq 
\tilde{\a}(L)
\end{equation}
and the parallel inequalities for $\d$. 
but this is of course inadequate to get the 
lower bound for $\tilde{\a}$. 

\begin{thm} 
    In the situation of Setting \ref{surfaces} 
    we claim the following. 
    \begin{itemize}
\item[$(1)$]
 If $X$ is abundant, 
 we have $\tilde{\a}(-K_X)\geq 1$ so that $-K_X+\e A$ 
 admits a cscK metric for any ample divisor $A$. 
 Otherwise, in the case $(a)$ one has  
 $\tilde{\d}(-K_X) \leq 1/2$, $\tilde{\a}(-K_X) \leq 1/3$ 
 and $-K_X+\e A$ never admits cscK for any $A$. 
 In the case $(b)$ one has the border line 
 $\tilde{\d}(-K_X) \leq 1$ and $\tilde{\a}(-K_X) \leq 2/3$, 
 while $-K_X+\e A$ actually admits a cscK metric for an arbitrary ample $A$. 
\item[$(2)$] 
If the nine points blowing up of $\P^2$ is 
equipped with an ellptic pencil, 
$\tilde{\a}(-K_X)\geq 1$ so that $-K_X+\e A$ admits cscK for any $A$. 
Otherwise $\tilde{\d}(-K_X) \leq 1$ and $\tilde{\a}(-K_X) \leq 2/3$. 
    \end{itemize} 
\end{thm} 

The last part of $(2)$ is not fully understood. 
It is interesting to ask whether for 
the very general nine points   
$-K_X +\e A$ admits a cscK metric. 

\begin{proof}
    $(1)$ 
The idea to estimate $\tilde{\a}$ has much in common 
with \cite{Che08} and \cite{Dervan}. 
The elliptic pencil $X \to \P^1$ 
enables us to take a smooth curve 
$C \in \abs{-K_X}$ which passes through any 
fixed point $p \in X$. 
Let us assume, to the contrary, $\tilde{\a}(-K_X) < 1$.
It implies that one may take $a< 1$ and $D \sim_\Q -K_X +\frac{1}{m}A$ 
for a sufficiently large $m \in \N$ such that 
$(X, aD)$ is not log canonical at some point $p \in X$. 
Let us write $D=tC +D'$ with $t \in \Q$, $C \nsubseteq \supp D'$.
If $H$ denotes the hyperplane line bundle of $\P^2$, 
one computes 
\begin{equation}
    D'.H =(D-tC)H
    =(-(1-t)K_X+\frac{1}{m}A).H
    =3(1-t) +\frac{1}{m}A.H. 
\end{equation} 
the left-hand side is nonnegative 
so we may assume $t \leq 1+\frac{c}{m}$ for some small constant $c$, 
and hence $ta <1$. 
Since $(X, C)$ is log canonical at $p$, 
by \cite{Wilson}, Corollary 6, 
$(X, \frac{1}{1-ta}(aD-taC))$ is not log canonical at $p$. 
It follows $\mult_p(\frac{a}{1-ta}D')>1$ and hence 
\begin{equation}
aC.D' \geq a \mult_pD' >1-ta.  
\end{equation} 
It implies 
\begin{equation}
    a \cdot \frac{1}{m}A.C 
    =aD.C =a(tC^2+D'C) 
    >1-ta. 
\end{equation}
After  $m \to \infty$ 
one deduces $a\geq 1$, which is a contradiction. 

In general to give a section of $\P E \to B$ 
is equivalent to give a quotient of $E$. 
If $C$ denotes the section corresponding to 
the quotient line bundle $Q$ in the case $(a)$, 
it is well known that 
$\abs{-mK_X}=\{ 2mC\}$ holds for any $m \in \N$. 
From this result, direct computation yields $\d(-K_X)=\a(-K_X)=1/2$. 
Even though the line bundle $L=-K_X$ is nef, 
the property $\tilde{\d} \geq \frac{n+1}{n}\tilde{\a}$ 
can be proved by reduction to the ample case. 
Now $\d \geq \tilde{\d} \geq \frac{n+1}{n}\tilde{\a}$ 
implies $\tilde{\a}(-K_X) \leq 1/3$. 
In the splitting case $(b)$ 
one has two sections $C_1$ and $C_2$ so that 
$\abs{-mK_X}=\{ m(C_1+C_2)\}$. 
In the same manner we obtain $\d(-K_X)=\a(-K_X)=1$ 
and $\tilde{\a}(-K_X) \leq 2/3$. 

Existence of cscK for the geometrically ruled surfaces 
has been completely studied by \cite{Fujiki}. 
In particular, the case $(a)$ does not admit any cscK 
and the case $(b)$ does. 
When $X$ is abundant, 
one can also utilize Theorem \ref{existence of cscK} 
to obtan the cscK metric. 

$(2)$ 
Using the elliptic pencil, 
the completely pararell argument shows 
$\tilde{\a}(-K_X) \geq 1$. 
In this case Theorem \ref{existence of cscK} 
shows the existence of cscK metrics. 
(Since $\tilde{\d}(-K_X) \geq 3/2$, 
one may also apply  
\cite{Zhang21}, Theorem 2.4 to this setting.)
If $C$ denotes the   
strict transform of the cubic passing 
the nine points, one has
$\abs{-mK_X}= \{ mC\}$ for any $m \in \N$.
We obtain $\d(-K_X)=\a(-K_X)=1$ 
and $\tilde{\a}(-K_X) \leq 2/3$. 

\end{proof}

Notice that $\d(-K_X) \geq \frac{n+1}{n} \a(-K_X)$ 
does not hold in the above example. 
One can also points out that 
$\a(-K_X) >\tilde{\a}(-K_X)$ happens in the 
non-abundant case of $(1)$. 
As far as the author knows, 
these phenomena have not concretely 
observed in the literatures. 
\section{Stability of the holomorphic tangent bundle}
\label{stability of the holomorphic tangent bundle}

As it was classicaly known to \cite{Tian92}, 
if a Fano manifold $X$ admits a K\"ahler-Einstein metric, 
the holomorphic tangent bundle and its 
canonical extenstion sheaf:  
\begin{equation}
    0 \to \cO_X \to E \to T_X \to 0
\end{equation} 
defined by $c_1(-K_X) \in H^1(X, \Omega_X^1)\simeq \Ext^1(T_X, \cO_X)$
is slope semistable with respect to the anti-canonical polarization. 

Let us briefly review the terminology and Tian's idea. 
Fix a compact K\"ahler manifold $(X, \omega)$. 
For a given connection $D$ of the holomorphic vector bundle $E$ 
we denote the curvature by $F$ and define the mean curvature $K \in \End(E)$ 
so as to satisfy  
\begin{equation}\label{HYM}
    \frac{\sqrt{-1}}{2\pi} F \wedge \omega^{n-1}
    =K \omega^n \id_E. 
\end{equation} 
A connection $D$ of the vector bundle $E$ 
is said to be Hermitian Yang-Mills (shortly HYM) 
with respect to the K\"ahler form $\omega$, 
if there exists a constant $\mu \in \R$ such 
that $K= \mu \id_E$. 
The constant $\mu$ is determined 
by the cohomology classes as 
\begin{equation}
    \mu = \mu(E)=\frac{\int_X c_1(E) \cup [\omega]^{n-1}}{\rank{E} \cdot \int_X \omega^n}. 
\end{equation} 
A Hermitian fiber metric on $E$ is called Hermite-Einstein 
if the associated Chern connection $D$ is HYM. 
More generally, a sequence of Hermitian fiber metrics 
$H_\e$ with K\"ahler forms $\omega_\e$ gives  
\emph{approximate Hermite-Einstein structure} 
if it satisfies 
\begin{equation}
    \abs{  K_\e -\mu \id_E } \leq \e. 
\end{equation} 
Note that the present definition is 
slightly different from \cite{Kob87} where 
the K\"ahler form $\omega$ on the base space is fixed. 

\begin{dfn}\label{def of slope stability}
    Fix a nef and big line bundle $L$ and 
     an ample lin bundle $A$. 
    For any $\cO_X$-submodule $S \subset E$ 
    we define the slope 
\begin{equation}
    \mu(S) := \frac{\int_X c_1(S) \cup c_1(L)^{n-1}}{\rank{S} \cdot  \int_X c_1(L)^n}
\end{equation} 
The vector bundle $E$ is called slope semistable 
    if $\mu(S) \leq \mu(E)$ holds for any $\cO_X$-submodule $S \subset E$.    
\end{dfn}

\begin{thm}[\cite{Kob87}, Theorem 5.8.6]\label{approximate HE implies semistability}
    If the vector bundle $E$ admits an approximate Hermite-Einstein structure 
    with respect to a sequence of K\"ahler metrics $\omega_\e$ in $c_1(L+\e A)$, 
it is slope semistable. 
\end{thm}

The proof is totally the same as \cite{Kob87}. 

Next we explain the extension of the vector bundles. Fix the vector bundle $S, Q$ on $X$ 
and take $\a$ be a $\bar{\partial}$-closed smooth $(0, 1)$-form 
represents the class $[\a] \in H^1(X, Q^\vee \otimes S) \simeq \Ext^1(Q, S)$. 
Here $Q^\vee$ denotes the dual vector bundle. 
It defines the short exact sequence:  
\begin{equation}
    0 \to S \to E \to Q \to 0. 
\end{equation} 
If one further takes a smooth fiber metric of $S, Q$, 
it defines the Chern connection $D_S, D_Q$ and 
the form $\a^*$ valued at the adjoint endomorphism. 

\begin{thm}[See \cite{Tian92}, Proposition 2.1]\label{curvature formula}
    The $(0, 1)$-part of the connection 
\begin{equation}
    D=
    \begin{pmatrix}
        D_S & \a \\
        \a^* & D_Q  
    \end{pmatrix}
\end{equation} 
on the $C^\infty$ vector bundle $S \oplus Q$ 
defines an integrable complex structure 
and hence holomorphic vector bundle $E$ 
which is canonically isomorphic to the extension 
defined by $[\a] \in H^1(X, Q^\vee \otimes S)$. 
If one denotes the induced connection of the homomorphism vector bundle $\Hom(Q, S)$
 by $D^{\Hom}$, 
 the Chern curvature is computed as 
 \begin{equation}
    F=
    \begin{pmatrix}
        F_S +\a \wedge \a^* & D^{\Hom(Q, S)}\a \\
        D^{\Hom(S, Q)}\a^* & F_Q+\a^* \wedge \a 
    \end{pmatrix}. 
 \end{equation} 
\end{thm}

Now if $X$ is a Fano manifold and $\omega=\omega_g$ is a K\"ahler-Einstein metric, 
we may consider $S=\cO_X$ endowed with the flat metric 
and $Q=T_X$ endowed with the K\"ahler-Einstein metric $g$. 
The multiple of the anti-canonical class $c_1(-K_X) \in H^1(X, Q^\vee \otimes S)$ 
defines the \emph{canonical extension} of $T_X$. 
For a fixed constant $a>0$ the extension is deifned by the $\Omega_X^1$-valued $(0,1)$-form
\begin{equation}
    \a := a \omega_g. 
\end{equation}

If one applies Theorem \ref{curvature formula}, it is 
easy to see that 
these dasum enjoy the HYM condition (\ref{HYM}) 
with $a =(n+1)^{-1/2}$. 
Therefore 
the existence of KE metric implies 
the semistablity not only of 
$T_X$ but also of its canonical extension sheaf. 

Let us finally consider the case when the anticanonical line bundle is 
nef and big. 
Thanks to Theorem \ref{existence of twisted KE}, 
if $\d(-K_X)>1$
we obtain the twisted K\"ahler Einstein metric $\omega=\omega_\e$ 
which satisfies $\Ric(\omega)=\omega +\e \theta$.   
It is natural to consider $S=\cO_X$ endowed with the flat metric, 
$Q=T_X$ endowed with twisted the K\"ahler-Einstein metric $g=g_\e$, 
and the $\Omega_X^1$-valued $(0,1)$-form $\a := a \omega_g$. 
It is obvious that $D^{\Hom} \a =0$. 
Since $g$ is twisted KE, we have the identity 
\begin{align}
    \frac{\sqrt{-1}}{2\pi} F_Q \wedge \omega^{n-1}
    =\frac{1}{n} (1+\frac{\e}{n}\Theta_\e) \omega^n \id_Q,  
\end{align}
where $\Theta_\e$ is the Hermitian endomorphism 
defined by $g_\e^{i\bar{k}}\theta_{j\bar{k}}$ in a local coordinate. 
The trouble here is that 
the choice of coordinates which diagonalize $\Theta_\e$ depends on $\e$. 
In our situation this point is settled by using the new result of \cite{DZ22}, Theorem 6.3 
which assures the existence of the K\"ahler-Einstein current; 
$\Ric(\omega_0) =\omega_0$. More precisely, 
for a fixed K\"ahler form $\a$ and any smooth representative $ \beta \in c_1(-K_X)$, 
we take a Ricci potential function $f$ such that 
$\Ric(\a)=\beta +\dd f$ holds.  
Then there exists a unique $\beta$-psh function $u$ with minimal singularity, 
which satisfies the degenerate complex Monge-Amp\`ere equation 
\begin{equation}\label{singular KE}
    (\beta +\dd u)^n = e^{-u+f} \a^n.  
\end{equation} 
Let us write $\omega_0 := \beta+\dd u$. 
Because the potential is bounded from above; $u \leq C$, 
the equation (\ref{singular KE}) tells that the non-pluripolar product $\omega_0^n$ is bounded from below by 
a smooth volume form.  
From the uniqueness, taking a subsequence if necessary, we observe that 
the twisted KE metrics $\omega_\e$ weakly converges 
to the KE current $\omega_0$. 
Since the potential $u$ has minimal singularity, 
one can see from the standard property of 
the non-pluripolar product (see \cite{GZ17}) 
that 
$\omega_\e^k$ converges to $\omega_0^k$ 
for any $1 \leq k \leq n$. 
With the above preparation, the identity 
\begin{align}
    (\Tr \Theta_\e^2) \omega_\e^n  
    =n(n-1) \theta \wedge \theta \wedge \omega_\e^{n-2}, 
\end{align} 
shows that the square sum of the eigenvalues of $\e \Theta_\e$ 
converges to zero. 
It yields $\e \Theta_\e \to 0$. 
To summurize the above, taking $a=(n+1)^{-1/2}$, 
$\omega_\e$ defines an approximate Hermite-Einstein structure 
of the vector bundle $E$, 
with respect to the K\"ahler metrics $\omega_\e$ themselves.  
Theorem \ref{approximate HE implies semistability} 
asserts that $E$ is slope semistable with respect to the nef and big line bundle 
$-K_X$. 
This completes the proof of Theorem \ref{slope semistability}.


\begin{thebibliography}{99.}%
    %
    %
    
    
    
    
    
    
    
       
    \bibitem{Berm13b}R.~J. Berman: 
       \newblock \emph{A thermodynamical formalism for Monge-Amp\`ere equations, Moser-Trudinger inequalities and K\"ahler-Einstein metrics}. 
       \newblock  Adv. Math. \textbf{248} (2013), 1254--1297. 	
    
    
    
    

    
    
    
    
    \bibitem{BBEGZ11}R.~J. Berman, S. Boucksom, P. Eyssidieux, V. Guedj, and A. Zeriahi: 
       \newblock \emph{K\"ahler-Einstein metrics and the K\"ahler-Ricci flow on log Fano varieties}. 
       \newblock to appear in J. Reine Angew. Math., \texttt{arXiv:1111.7158}.  
    
    \bibitem{BBJ15}R.~J. Berman, S. Boucksom, and M. Jonsson: 
        \newblock \emph{A variational approach to the Yau-Tian-Donaldson conjecture}. 
        \newblock J. Amer. Math. Soc. \textbf{34} (2021), 605--652
        
    \bibitem{BDL15}R.~J. Berman, T. Darvas, and C.~H. Lu: 
        \newblock \emph{Convexity of the extended K-energy and the large time behavior of the weak Calabi flow}. 
        \newblock Geometry and Topology, \textbf{24}, 1907--1967 (2020)      
    
    
    
    
    
    \bibitem{BHJ17}
    S. Boucksom T. Hisamoto and M. Jonsson: 
    \newblock \emph{Uniform K-stability, Duistermaat-Heckman measures and singularities of pairs}. 
    \newblock Ann. Inst. Fourier (Grenoble) \textbf{67} no. 2 (2017), 743--841. 
    
    \bibitem{BHJ19}
    S. Boucksom T. Hisamoto and M. Jonsson: 
    \newblock \emph{Uniform K-stability and asymptotics of energy functionals in K\"ahler geometry}. 
    \newblock J. Eur. Math. Soc.. \textbf{21} no. 9 (2019), 2905--2944
    
    
    \bibitem{BJ17}H.~Blum and M.~Jonsson: 
    \newblock \emph{Thresholds, valuations, and K-stability}.
    \newblock \texttt{arXiv:1706.04548v1}. 
    
    
    
    
     

    \bibitem{Che08}I.~ Cheltsov: 
      \newblock \emph{Log canonical thresholds of del Pezzo surfaces}. 
      \newblock Geometric and Functional Analysis \textbf{11} (2008), 1118--1144. 

    
    \bibitem{Chen00b}X. Chen: 
            \newblock \emph{On the lower bound of the Mabuchi energy and its application}.                    
            \newblock Int. Math. Res. Not. {\bf 4} (2000), no. 12, 607--623. 
    
    \bibitem{CC21a}X.~Chen and J.~Cheng: 
       \newblock \emph{On the constant scalar curvature K\"ahler metrics, a priori estimates}. 
       \newblock J. Am. Math. Soc., \textbf{34} (2021), 909--936. 

    \bibitem{CC21b}X.~Chen and J.~Cheng: 
      \newblock \emph{On the constant scalar curvature K\"ahler metrics, existence results}. 
     \newblock J. Am. Math. Soc., \textbf{34} (2021), 937--1009.

    
    
    
    \bibitem{CO75}B.~Y. Chen and K.~Ogiue:  
    \newblock \emph{Some characterization of complex space forms in terms of Chern classes}. 
    \newblock  Quart. J. Math., \textbf{26} (1975), 119--121. 

    \bibitem{Chen21}G.~Chen: 
    \newblock \emph{The J-equation and the supercritical deformed Hermitian-Yang-Mills equation}. 
    \newblock Invent. Math. \textbf{225}  (2021), Pages 529--602. 

    
    
    

    
    
    
    
    
    
    
     
    \bibitem{DZ22}T.~Darvas and K.~Zhang: 
      \newblock \emph{Twisted K\"ahler-Einstein metrics in big classes}. 
      \newblock \texttt{arXiv:2208.08324}. 

    
    
    
    \bibitem{Dervan}
    R.~Dervan: 
    \newblock \emph{Alpha invariants and K-stability for general polarizations 
    of Fano Varieties}. 
    \newblock Int. Math. Res. Notices \textbf{16} (2015), 7162--7189. 
    
    
      

     
     
     
     
    \bibitem{Don99}S.~K. Donaldson: 
      \newblock \emph{Moment maps and diffeomorphisms}. 
      \newblock Asian J. Math., \textbf{3} (1999) no. 1, 1--15. 

     
    
    
   \bibitem{DGP} S. Druel, H. Guenancia, and M. P\u{a}un: 
    \newblock \emph{A decomposition theorem for $\Q$-fano K\"ahler-Einstein varieties}
    \newblock \texttt{arXiv:2008.05352}. 

    \bibitem{Dy20}Z. S. Dyrefelt: 
     \newblock \emph{Existence of cscK metrics on smooth minimal models}. 
     \newblock \texttt{arXiv:2004.02832}. 

    \bibitem{Fujiki}A. Fujiki: 
     \newblock \emph{Remarks on extremal K\"ahler metrics on ruled manifolds}. 
     \newblock Nagoya Math J. \textbf{126} (1992) 89--101.  

    \bibitem{FO16}K. Fujita and Y. Odaka: 
        \newblock \emph{On the K-stability of Fano varieties and anticanonical divisors}. 
        \newblock Tohoku Math. J. \textbf{70}(4) 70(4), 511--521 (2018). 
    

    
    \bibitem{GKP20}D.~Greb, S.~Kebekus, and T.~Peternell: 
       \newblock \emph{Projective flatness over klt spaces and uniformisation of varieties with nef
       anti-canonical divisor}. 
       \newblock J. Algebraic Geom. \textbf{31} (2022), 467--496. 

       \bibitem{GKPT19}D.~Greb, S.~Kebekus, T.~Peternell, and B. Taji: 
       \newblock \emph{The Miyaoka-Yau
       inequality and uniformisation of canonical models}. 
       \newblock Ann. Sci. Ec. Norm. Sup\'er., 
       \textbf{52} (2019), 1487--1535.


      
      
    \bibitem{GZ17}V. Guedj and A. Zeriahi: 
        \newblock \emph{Degenerate complex Monge-Amp\`ere equations}.
        \newblock EMS Tracts in Mathematics \textbf{26} (2017).    
    
      
      
     \bibitem{Hat21}M.~Hattori: 
     \newblock \emph{A decomposition formula for J-stability and its applications}. 
     \newblock \texttt{arXiv:2103.04603}. 

      
    
    
    
    
    


    
\bibitem{Kob87}S.~Kobayashi: 
     \newblock \emph{Differential geometry of complex vector bundles}. 
     \newblock Princeton University Press, 1987. 


    
    
    
    \bibitem{Liu20}W.~Liu: 
       \newblock \emph{The Miyaoka-Yau inequality on smooth minimal models}. 
       \newblock Bulletin of the London Mathematical Society \textbf{55} (3), 1196--1202. 

    \bibitem{Mab86}T.~Mabuchi: 
     \newblock \emph{K-energy maps integrating Futaki invariants.}
     \newblock Tohoku Math. J. (2) \textbf{38} (1986), no. 4, 575--593. 
    
    
    \bibitem{Miyaoka77}Y.~Miyaoka: 
        \newblock \emph{On the Chern numbers of surfaces of general type}. 
        \newblock Invent. Math. \textbf{42} (1977), no. 1, 225--237.   

    \bibitem{N90}A.~M. Nadel: 
       \newblock \emph{Multiplier Ideal Sheaves and K\"ahler-Einstein Metrics of Positive Scalar Curvature}. 
       \newblock Ann. of Math. \textbf{132} (1990), no. 3, 549--596. 
       
    
    
    
        
    
        
    
    
    
    

    
    
    
    \bibitem{SW08}J.~Song and B.~Weinkove: 
     \newblock \emph{On the convergence and singularities of the J-flow with applications 
     to the Mabuchi energy}. 
     \newblock Commun. Pure Appl. Math. \textbf{61} (2008), 210--229. 
    
    \bibitem{Song20}J.~Song: 
     \newblock \emph{Nakai-Moischezon criterions for complex Hessian equations.} 
     \newblock \texttt{arXiv:2012.07956}. 

        
    
    
    
    \bibitem{Tian92}G. Tian: 
     \newblock \emph{On stability of the tangent bundles of Fano varieties}. 
     \newblock Internat. J. Math. \textbf{3} (1992), no. 3, 401--413.

    \bibitem{Tian97}G. Tian: 
    \newblock \emph{K\"ahler-Einstein metrics with positive scalar curvature}.
    \newblock Inv. Math. \textbf{130} (1997), 239--265. 
    
    
    \bibitem{Wilson}A. Wilson: 
     \newblock \emph{Smooth exceptional del Pezzo surfaces}. 
     \newblock PhD Thesis, University of Edinburgh, 2010. 

    
     \bibitem{Xu}C. Xu: 
     \newblock \emph{K-stability for varieties with a big anticanonical class}. 
     \newblock \texttt{arXiv:2210.16631}. 

    
    \bibitem{Yau77}S.~T. Yau: 
         \newblock \emph{Calabi's conjecture and some new results in algebraic geometry}.   
         \newblock Proc. Nat. Acad. Sci. U.S.A \textbf{74} (1977), no. 5, 1798--1799. 
     
    \bibitem{Yau78}S.~T.~Yau: 
          \newblock \emph{On the Ricci curvature of a compact K\"ahler manifold and the complex Monge-Amp\'ere equation I}. 
          \newblock Com. Pure and Appl. Math., \textbf{31} (1978), 339--411. 
    
    \bibitem{Zhang20}K.~Zhang: 
          \newblock \emph{Continuity of delta invariants and twisted K\"ahler-Einstein metrics}. 
          \newblock Adv. Math. \textbf{388} (2021)

    \bibitem{Zhang21}K. Zhang: 
         \newblock \emph{A quantization proof of the uniform Yau-Tian-Donaldson conjecture}. 
         \newblock to appear in the Journal of the European Mathematical Society. 
               
    
\end{thebibliography}
\end{document}